\documentclass[11pt,twoside]{amsart}
\usepackage[all]{xy}
        \usepackage {amssymb,latexsym,amsthm,amsmath,mathtools}
        \usepackage{enumerate,color}
        
\usepackage[colorlinks, linkcolor=blue, citecolor=blue]{hyperref}
\usepackage{mathtools}
 \usepackage{cleveref}

\long\def\symbolfootnote[#1]#2{\begingroup%
\def\thefootnote{\fnsymbol{footnote}}\footnote[#1]{#2}\endgroup}

\newcommand{\diag}{\textup{diag}}

\makeatletter
\def\imod#1{\allowbreak\mkern10mu({\operator@font mod}\,\,#1)}
\makeatother

\newcommand{\qq}{\mathbb Q}

\newcommand{\cc}{\mathbb C}

\newcommand{\nat}{\mathbb{N}} 
\newcommand{\zz}{\mathbb Z}

\newcommand{\mg}[1]{{#1}^{\times}}

\newcommand{\ovl}{\overline}

\newcommand{\car}{\mathrm{char}}

\newcommand{\End}{\mathrm{End}}

\newcommand{\tr}{\mathrm{tr}}
\newtheorem{theorem}{Theorem}[section]
\newtheorem{lemma}[theorem]{Lemma}
\newtheorem{corollary}[theorem]{Corollary}
\newtheorem{proposition}[theorem]{Proposition}
\newtheorem*{theorem*}{Theorem}
\theoremstyle{definition}
\newtheorem{definition}[theorem]{Definition}
\newtheorem{remark}[theorem]{Remark}
\newtheorem{example}[theorem]{Example}
\numberwithin{equation}{section}

\newcommand{\ignore}[1]{}

\newcommand{\mynote}[1]{}

\begin{document}
\setcounter{section}{0}
\title{Splitting fields of differential symbol algebras}
\author{Parul Gupta}
\email{parulgupta1211@gmail.com}
\author{Yashpreet Kaur}
\email{yashpreetkm@gmail.com }
\author{Anupam Singh}
\email{anupamk18@gmail.com}
\address{IISER Pune, Dr. Homi Bhabha Road, Pashan, Pune 411 008, India}
\thanks{The second named author is supported by NBHM (DAE, Govt. of India): 0204/16(33)/2020/R$\&$D-II/26. The first and third named authors are funded by SERB through CRG/2019/000271 for this research.}
\today


\begin{abstract}
For $m\geq 2$, we study derivations on symbol algebras of degree $m$ over fields with 
characteristic not dividing $m$.
A differential central simple algebra over a field $k$ is split by a finitely generated extension of $k$.
For certain derivations on symbol algebras, we provide explicit construction of differential splitting fields and give bounds on their algebraic and transcendence degrees.
We further analyze maximal subfields that split certain 
differential symbol algebras.
\end{abstract}

\maketitle

\section{Introduction}

Derivation maps have been used as a prime tool in the study of simple algebras; for example see ~ Hochschild \cite{Ho}, Amitsur \cite{Ami1}, \cite{Ami2}, Hoechsmann \cite{Hoe}. The exposition in these references lead to the intriguing study of derivations on central simple algebras (see Amitsur \cite{Ami}, Juan and Magid \cite{LJARM}, Kulshrestha and Srinivasan \cite{AKVRS}) and on the Brauer groups (see Hoobler \cite{Hoo} and Magid \cite{Mag}).
In this paper, we investigate symbol algebras with derivations.

Let $k$ be a field and $\delta$ be a derivation on $k$. 
For $m\geq 2$, assume that $k$ contains a primitive $m$th root of unity  $\omega$. 
For  $\alpha, \beta \in \mg k$, the \textbf{symbol algebra} $A = ( \alpha, \beta)_{k, \omega}$ is an $m^2$-dimensional $k$-algebra generated by $u,v \in A$ satisfying the relations $u^m=\alpha, v^m =\beta$ and $vu = \omega uv$.
There is a unique derivation $d_s$ on $A$ having the properties that $d_s$ restricts to $\delta$ on $k$, $d_s(k(u)) \subseteq k(u)$ and $d_s(k(v)) \subseteq k(v)$ (see \Cref{derchar}).  
We show that any derivation on $A$ such that $d$ restricts to $\delta$ on $k$ is a sum of the \emph{standard derivation} $d_s$ and an inner derivation $\partial_{\vartheta}$ determined by a unique trace zero element $\vartheta \in A$ (see \Cref{derchar}).

Amitsur \cite[Theorem]{Ami} showed that for a central simple $k$-algebra $A$ and a maximal subfield $L$ of $A$, there exists a derivation $d$ on $A$ that restricts to $\delta$ on $k$ such that $d(L)\subseteq L$. 
For a differential symbol algebra $(A,d)$, we study subfields of $A$ with the property $d(L)\subseteq L$.
When $\delta$ is the zero derivation, we show that a maximal subfield with this property exists (\Cref{contantscentralizerequality}).
We further construct a derivation $d$ on the  matrix algebra $M_m(k)$ for which $d(L) \not \subset L$ for any cyclic subfield $L$ of $M_m(k)$ (\Cref{nocyclic_differentialfields}).

Juan and Magid \cite{LJARM} introduced the notion of differential splitting fields in the context of differential central simple $k$-algebras. 
Our main concern in this article 
is to construct differential splitting fields for differential symbol algebras.

We show an arbitrary differential symbol algebra $(A, d)$ is split by a finitely generated differential field extension over $k$.
We further prove that the differential algebra $(A, d_s)$ is split by  a  differential field extension of degree $m^2$ if $m$ is odd and of degree $2m^2$ if $m$ is even (see \Cref{split-standard-der}). 
In \Cref{zero-der-ext-split}, we consider the case where  $\delta$ is the zero derivation on $k$.
	In this case $d_s$ is also the zero derivation on $A$ and  $d = \partial_\vartheta$ for some $\vartheta\in A$. 
	If $\vartheta$ is such that $k(\vartheta)/k$ is a cyclic extension of degree $m$, we construct a differential field extension of transcendence degree $m$ which splits the algebra $(A,\partial_\vartheta)$.
Further, if $\vartheta^m \in k$ and $m$ is even then we obtain a differential field extension of transcendence degree $\frac{m}{2}$ which splits the algebra $(A,\partial_\vartheta)$. 
We conclude the paper by studying the maximal subfields that split the algebra $(A,d_s)$ (see \Cref{max-subfield-split}).

In the case of quaternion algebras ($m=2$), we retrieve results from \cite{AKVRS}(see \Cref{Quatex}).
Thus, our work generalizes a part of their work.

\subsection*{Notation} Throughout this paper we fix $m \in \nat$, a field $k$  containing a primitive $m$th root of unity, $\omega$, an algebraic closure  $\ovl{k}$ of $k$.
We denote the multiplicative group of $k$ and $k^m$ by $k^{\times}$ and ${k^{\times}}^{m}$, respectively.
For  $\alpha, \beta \in \mg k$, we  fix the symbol algebra $A = (\alpha, \beta)_{k, \omega}$ with generators $u,v$ such that $u^m =\alpha, v^m =\beta$ and $vu = \omega uv$.
We denote the diagonal matrix with diagonal entries $\lambda_1, \dots, \lambda_m \in k$ by   $\diag(\lambda_1, \dots, \lambda_m)$ . 
The $(m \times m)$-identity and zero matrices are denoted by $I_m$ and $\mathbf{0}_m$.

\section{Differential central simple algebras}\label{dcsa}

Let $R$ be a ring (need not be commutative) with identity. 
An additive map $d\colon R \rightarrow R$ is 
said to be a \textbf{derivation on $R$} if $d(xy)=xd(y) + d(x)y$ for all $x, y \in R$.
 A pair $(R, d)$ is said to be a \textbf{differential ring (resp. algebra or field)} if $R$ is a ring 
(resp. algebra or field) together with a derivation $d$ on $R$.
The set $C_{R,d} = \{ x\in R \mid d(x) =0\}$ is called the set of \textbf{constants of $(R,d)$}. 
Let $S/R$ be a ring (resp. algebra or field) extension and let $d$ be a derivation on $R$. 
A derivation $d_S$ on $S$ is called an \textbf{extension of $d$} if $d_S$ restricts to $d$ on $R$,
and $(S,d_S)$ is called a \textbf{differential ring (resp. algebra or field) extension of $(R,d)$.} 
The set of all derivations on $S$ that are extensions of $d$ is denoted by $Der(S/(R,d))$.

We recall here some examples of extensions of derivations that are useful in later sections.

\begin{example} Let $(k,\delta)$ be a differential field.
\begin{enumerate}[$(a)$]
\item If $E/k$ is a finite field extension then there is a unique extension of $\delta$ on $E$ (see \cite[Section 3]{Ros}). 
\item Let $B$ be a central simple $k$-algebra. 
For any element $\vartheta \in B$, 
there is an \textbf{inner derivation} $\partial_\vartheta $ on $B$ given by $\partial_\vartheta (x)= x\vartheta -\vartheta  x$ for $x\in B$. 			
Note that $\partial_\vartheta (k) =0$.
\item	The map $\delta^c : M_m(k) \rightarrow M_m(k)$ defined by $\delta^c ((a_{ij})) = (\delta(a_{ij}))$ is a derivation on $M_m(k)$
and $\delta^c$ is called the \textbf{coordinate-wise derivation on $M_m(k)$}. 
Then $(M_m(k), \delta^c)$ is a differential algebra over $(k, \delta)$.
\end{enumerate} 
\end{example}

Let $(k, \delta)$ be a differential field.
Let $B$ be a finite-dimensional central simple $k$-algebra.
For $d\in Der(B/(k,\delta))$, we call $(B,d)$ a \textbf{differential central simple algebra over $(k,\delta)$}.
Amitsur \cite{Ami} showed that if $d,d'\in  Der(B/(k,\delta))$, then $d' =d+ \partial_\vartheta $ for some $\vartheta \in B$.
In particular, one derivation on a central simple $k$-algebra $B$ determines the whole set $Der(B/(k,\delta))$. 
For the matrix algebra $M_m(k)$, we have $ Der(M_m(k)/(k, \delta)) = \{d_P =\delta^c + \partial_P \mid\ P\in M_m(k)\}$ (see \cite[Theorem 2]{Ami}).

Let $(B_1,d_1)$ and $(B_2,d_2)$ be two differential central simple algebras over $(k,\delta)$.
Then the algebra $(B_1\otimes_k B_2, d_1\otimes d_2)$ is again a differential central simple algebra over $(k,\delta)$, where $(d_1\otimes d_2)(b_1\otimes b_2) =  d_1(b_1)\otimes b_2 + b_1\otimes d_2(b_2)$.
A map $\phi: (B_1,d_1) \rightarrow (B_2,d_2)$ is called a \textbf{differential isomorphism} if $\phi: B_1\rightarrow B_2$ is a $k$-algebra isomorphism such that $\phi \circ d_1 = d_2 \circ \phi$.
We say that $(B_1,d_1)$ and $(B_2,d_2)$ are \textbf{differential isomorphic} if there exists a differential isomorphism $\phi: (B_1,d_1) \rightarrow (B_2,d_2)$.

Another important notion in the theory of central simple $k$-algebra is the notion of a splitting field. 
We recall that the central simple $k$-algebra $B$ \textbf{splits over a field extension $E/k$} in usual sense if $B\otimes_k E\simeq M_m(E)$ as $E$-algebras and $E$ is called a \textbf{splitting field} of $B$. 
An analogous notion of a differential splitting field of a differential central simple algebra was introduced by Juan and Magid \cite{LJARM}.  
A differential central simple algebra $(B,d)$ over $(k,\delta)$ \textbf{splits over a differential field $(E,\delta_E)\supseteq (k,\delta)$} if $(B\otimes_k E, d^{\ast} = d\otimes \delta_E) \simeq (M_m(E),\delta_E^c)$ as differential $E$-algebras.
One of the main tools that is used to construct the differential splitting fields is the following result from \cite[Proposition 2]{LJARM}:

\begin{proposition}\label{JM_split}
Let $P\in M_m(k) \setminus \{0\}$.  Then differential algebras $(M_m(k), d_P)$ and $(M_m(k), \delta^c)$ are differential isomorphic if and only if there exists $F \in GL_m(k)$ such that $\delta^c(F) =PF$. 
\end{proposition}

It is well-known that a central simple $k$-algebra of degree $m$ is split by a finite field extension of $k$ of degree at most  $m$. 
However, there are differential central simple algebras that are not split by any algebraic extension as seen in \Cref{Quatex} (\ref{Quatex3}) below.

In this paper, the main object of interest is differential symbol algebras.
The symbol algebra $A= (\alpha,\beta)_{k,\omega}$ is a 
 central simple algebra of degree $m$ (see~\cite[Chapter 11, Theorem 
1]{Dra}), and as a $k$-vector space we have 
$$A = \bigoplus_{0\leq i,j\leq m-1} k u^i v^j, \mbox{ where } u^m=\alpha,\, v^m=\beta \mbox{ and } vu = \omega uv.$$
For $m=2$, $-1$ is the only primitive $2$nd root of unity, thus the algebra $(\alpha, \beta)_{k,-1}$ is denoted just by $(\alpha, \beta)_k$.
Note that, when $m=2$, these are all \textbf{quaternion algebras}, i.e.~$4$-dimensional central simple $k$-algebras.

\begin{example}\label{splittingex}\cite[p.~77-78]{Dra}
Let	$A=(\alpha,\beta)_{k,\omega}$ be a symbol algebra with generators $u,v$.
For $\xi \in \ovl{k}$ such that $\xi^m = \alpha$, the homomorphism $\mathbf{\Phi_{A, \xi}: A\otimes_kk(\xi) \rightarrow M_m(k(\xi))} $
determined by
\begin{center}
$u \otimes 1 \mapsto \mathcal A = \diag(\xi, \omega^{m-1}\xi, \omega^{m-2}\xi, \ldots,\omega\xi ), \,\,\,\,
v \otimes 1 \mapsto \mathcal B = \left(
\begin{matrix}
& \beta \\
I_{m-1} &\\
\end{matrix} \right),$
\end{center}
and $\lambda\otimes r \mapsto \lambda r I_m$ where $\lambda \in k$ and $r\in k(\xi)$ is an isomorphism of $k(\xi)$-algebras.
In particular, $k(\xi)$ splits the $k$-algebra $A$.

The isomorphism $\Phi_{A, \xi}$ plays an important role in constructing certain differential splitting fields of differential symbol algebras.

\end{example}

For $a\in A$, the map $l_a: A \rightarrow A$, $x \mapsto ax$ defines a linear transformation and \textbf{trace of $a$}, denoted by $\tr(a)$, is the trace of the map $l_a$. 
For an element $ a = \sum_{0\leq i, j 
				\leq m-1} a_{ij}u^{i}v^j$, the coefficient of the basis element $u^rv^s$ in $a\cdot u^rv^s$ is $a_{00}$ and hence trace of $a$ is $m^2a_{00}$.
Thus the set of all trace zero elements of $A$ is
$$A^0 := \bigoplus_{\substack{0\leq i,j\leq m-1\\ i+j >0}} k u^i v^j\,.$$

We conclude this section by recalling some results from \cite{AKVRS} for quaternion algebras that are generalized in this paper to symbol algebras.

\begin{theorem}[Quaternion Algebra] \label{Quatex}
Let $(k,\delta)$ be a differential field  of characteristic $0$. 
Let $Q =(\alpha, \beta)_k$ be a quaternion algebra.

\begin{enumerate}
\item \cite[Proposition 3.2]{AKVRS} Let $d_s \in Der(Q/(k,\delta))$  be such that $d_s(u) = \frac{\delta(\alpha)}{2\alpha}u$ and $d_s(v)= \frac{\delta(\beta)}{2\beta}v$, where $u^2 = \alpha$ and $v^2 =\beta$.
Then, for every derivation  $d\in Der(Q/(k, \delta))$, there exists  a unique $\vartheta\in Q^0$ such that $d = d_s+\partial_\vartheta$.

\item \cite[Theorem 4.1]{AKVRS} \label{Quatex2} There is a differential field extension $(E,\delta_E)$, where $E$ is finitely 
generated as a field over $k$, that splits the differential quaternion algebra $(Q,d)$. 

\item \cite[Theorem 4.5]{AKVRS} \label{Quatex3} If $\delta$ is the zero derivation and $Q$ is a division algebra with non-zero derivation $d$ which extends $\delta$, then $(Q,d)$ is split by a field extension of transcendence degree $1$ and is not split by any algebraic extension of $k$.   
\end{enumerate}
\end{theorem}

In \cite{AKVRS}, the characterization of  derivation in terms of $d_s$ along with \Cref{JM_split} and \Cref{splittingex} is used in order to construct the splitting fields in $(\ref{Quatex2})$ and $(\ref{Quatex3})$. 
We will follow the same strategy to obtain results for symbol algebras.

\section{Derivations on symbol algebras}

We begin by characterizing derivations on symbol algebras over a differential field $(k,\delta)$.

\begin{proposition} \label{derchar} 
Let $(k,\delta)$ be a differential field.
Let $A = (\alpha, \beta)_{k,\omega}$ be a symbol algebra of degree $m$. 
\begin{enumerate} 
\item  \label{derchar1} A map $d \colon A\rightarrow A$ belongs to  $Der(A/(k, \delta))$ if and only if there are elements
$a_{ij}, b_{ij} \in k$ for $0 \leq i,j \leq m-1$ satisfying the following conditions:  
\begin{enumerate}
\item $a_{10} = \frac{\delta(\alpha)}{m\alpha}$   and $a_{i0} =0$ for $i\neq 1$.  
\item $b_{01} = \frac{\delta(\beta)}{m\beta}$  and $b_{0j} =0$ for $j\neq 1$.  
\item

Let $\mathcal{A}= (a_{ij})_{0\leq i,j\leq m-1}$ and $\mathcal B = (b_{ij})_{0\leq i,j\leq m-1}$. 
Let $\mathcal{A}'$ and  $\mathcal B'$ be the minors of $a_{10}$ and $b_{01}$ respectively. 
Then $T_\alpha^{-1}\mathcal{B}'T_{\beta} = -\mathcal{A}'$ where for $\gamma\in \mg k$, 
{\small{$$ T_\gamma= \left(\begin{matrix} 0&0& \ldots &0&\frac{1-\omega}{\gamma}\\
\omega^{2}-\omega&0&\ldots& 0 & 0 \\ 0&\omega^{3}-\omega  & \ldots & 0&0\\
\vdots & \vdots & \ddots  & \vdots & \vdots \\ 0&0& \ldots & \omega^{m-1}-\omega &0 \end{matrix} 
\right).
$$}}

\item $d(u) = \displaystyle \sum_{0\leq i, j 
\leq m-1} a_{ij}u^iv^j$ and $d(v) = \displaystyle \sum_{0\leq i, j \leq m-1 } b_{ij} u^iv^j$.
\end{enumerate}
\item \label{derchar2} Let $d_s \in Der(A/(k, \delta))$ be such that $d_s(u) = 
\frac{\delta(\alpha)}{m\alpha}u $ and $ d_s(v) =\frac{\delta(\beta)}{m\beta} v$. 
Then for every $d\in 
Der(A/(k, \delta)) $ there exists a unique element $\vartheta \in A^0$ such that $d =d_s+\partial_\vartheta$.
\end{enumerate} 
\end{proposition}

\begin{proof}	
$(1)$ Let  $d\in  Der(A/(k, \delta)) $.
Since $vu = \omega uv$, we get
\begin{equation}\label{dervu}
d(v)u+vd(u)=\omega(d(u)v+ud(v)).
\end{equation}
For $0\leq i,j\leq m-1$, let $a_{ij}, b_{ij}\in k$ be such that $d(u) = \sum_{0\leq i, j 
\leq m-1} a_{ij}u^iv^j$ and $d(v) = \sum_{0\leq i, j \leq m-1 } b_{ij} u^iv^j$.
Substituting these values of $d(u)$ and $d(v)$ in \Cref{dervu} and using the relation $vu = \omega uv$, we obtain		
$$ \sum_{0\leq i, j 
				\leq m-1} b_{ij}(\omega^j-\omega)u^{i+1}v^j+ \displaystyle \sum_{0\leq i, j \leq m-1 } a_{ij}(\omega^i-\omega) u^iv^{j+1}=0\,.$$
Comparing the coefficients of $u^iv^j$ for $0\leq i, j \leq m-1$, we obtain the following relations:
\begin{align}
\label{ab}      a_{0(m-1)}\beta	+b_{(m-1)0}\alpha&=0,\\
\label{a00}	    a_{0(j-1)}(1-\omega)+b_{(m-1)j}(\omega^j-\omega)\alpha&=0\qquad \text{for}\ \ 1\leq j\leq m-1,\\
\label{b00}	    a_{i(m-1)}(\omega^i-\omega)\beta+b_{(i-1)0}(1-\omega)&=0\qquad \text{for}\ \ 1\leq i\leq m-1,\\
\label{a00b00}  a_{i(j-1)}(\omega^i-\omega)+b_{(i-1)j}(\omega^j-\omega)&=0 \qquad \text{for}\ \ 1\leq i,j\leq m-1.
\end{align}
Equations \ref{a00}, \ref{b00} and \ref{a00b00} further imply that $a_{i0}= b_{0j}=0$  for $i,j\in \{0,\dots,m-1\}\setminus\{1\}$. 
Equations \ref{ab} and \ref{a00b00} give the condition (c).	
Since $u^m=\alpha$, we have 
$$d(u^m)  =  \sum_{0\leq l \leq m-1} u^{l} d(u) u^{m-l-1} = \delta(\alpha).$$
Substituting $d(u) = \displaystyle \hspace{-0.4cm}\sum_{0\leq i, j \leq m-1} \hspace{-0.4cm}a_{ij}u^iv^j$ and  comparing the coefficients of $u^iv^j$ for $0\leq i, j \leq m-1$, we obtain 
$$ma_{10}\alpha=\delta(\alpha) \mbox{ and } a_{i0}=0  \mbox{ for } i\in \{0,\dots,m-1\}\setminus\{1\},$$
and $a_{ij} \in k$ for $0\leq i\leq m-1, 1\leq j\leq m-1$.

Similar computations for the relation $y^m=\beta$, yield us that 
$$mb_{01} \beta =\delta(\beta)  \mbox{ and } b_{0j}=0  \mbox{ for } j\in \{0,\dots,m-1\}\setminus\{1\},$$
and $b_{ij} \in k$ for $1\leq i\leq m-1, 0\leq j\leq m-1$.
Thus we conclude that if  $d \in Der(A/(k, \delta))$, then the conditions $(a),(b)$ and $(c)$ hold.

Conversely, let  $a_{ij}, b_{ij}\in k$ for $0 \leq i,j \leq m-1$  satisfying  the conditions $(a), (b), (c)$ and $(d)$. 
Define a map $d \colon A\rightarrow A$ such that 
$$d(x) =d \left( \displaystyle \sum_{0\leq i, j \leq m-1} \hspace{-0.4cm} x_{ij}u^iv^j\right) 
=  \displaystyle \sum_{0\leq i, j \leq m-1}\hspace{-0.4cm} \delta (x_{ij})u^iv^j + \displaystyle \sum_{0\leq i, j \leq m-1} \hspace{-0.4cm} x_{ij} d(u^iv^j),$$
where $d(u) = \displaystyle \hspace{-0.3cm} \sum_{0\leq i, j 
\leq m-1}\hspace{-0.3cm} a_{ij}u^iv^j, \,d(v) = \displaystyle \hspace{-0.3cm}\sum_{0\leq i, j \leq m-1 }\hspace{-0.3cm} b_{ij} u^iv^j$, and for $i,j \in \{0, \ldots, m-1\}$ 
$d(u^iv^j) =d(u^i)v^j+ u^id(v^j)$ with  $\displaystyle d(u^i) = \hspace{-0.3cm}\sum_{0\leq l \leq i-1} \hspace{-0.3cm}u^{l} d(u) u^{i-l-1}, d(v^j) = \hspace{-0.3cm}\sum_{0\leq l \leq j-1} \hspace{-0.3cm}v^{l} d(v) v^{j-l-1} $. 
The hypotheses on $a_{ij}'s, b_{ij}'s$ imply that $d(u^m) = \delta(\alpha), d(v^m) = \delta(\beta)$ and $d(vu) = d(\omega uv)$. 
Thus $d$ is well-defined. Clearly, $d$ is additive, by definition. 
It is easy to check that Leibniz rule is satisfied for the products of basis elements $u^iv^j$.
Therefore, by the additivity of $d$, Leibniz rule holds for arbitrary elements.

$(2)$ Let $\displaystyle {\vartheta = \sum_{0\leq i,j\leq m-1}\vartheta_{ij}u^iv^j  \in A}$ 
be such that $d = d_s +\partial_{\vartheta}$.
We have
$$ d(u) = d_s(u)+\partial_{\vartheta}(u) = \frac{\delta(\alpha)}{m\alpha}u  + u{\vartheta}-{\vartheta}u\,.$$
Thus
\begin{align*}
u{\vartheta}-{\vartheta}u &=\sum_{0\leq i,j\leq m-1}a_{ij}u^iv^j - \frac{\delta(\alpha)}{m\alpha}u  \\
&=  \sum_{0\leq j\leq m-1}(1 -\omega^j)\vartheta_{(m-1)j}\alpha v^j
+ \sum_{\substack {1\leq i\leq m-1\\0\leq j\leq m-1}}(1 -\omega^j)\vartheta_{(i-1)j}u^iv^j\,.
\end{align*}
For $1 \leq j\leq m-1$, we obtain 
$$ \left( \begin{matrix} \vartheta_{0j} \\ \vdots \\ \vartheta_{(m-2)j}\\ \vartheta_{(m-1)j} \end{matrix} \right)=
 \frac{1}{1- \omega^j} \left( \begin{matrix} a_{1j} \\ \vdots \\ a_{(m-1)j}\\ a_{0j}/\alpha
\end{matrix} \right).$$
From the equation 
$d(v) = d_s(v)+\partial_{\vartheta}(v) = \frac{\delta(\beta)}{m\alpha}v  + v{\vartheta}-{\vartheta}v$, 
we obtain 
$$ \left( \begin{matrix} \vartheta_{10} \\ \vdots \\ \vartheta_{(m-2)0}\\ \vartheta_{(m-1)0} \end{matrix} \right)=
 \left( \begin{matrix}  \frac{b_{11}}{\omega-1} \\ \vdots \\ \frac{b_{(m-2)1}}{\omega^{m-2}-1}\\
\frac{b_{(m-1)1}}{\omega^{m-1}-1} \end{matrix} \right).$$
Thus
$${\vartheta} = \vartheta_{00} + \sum_{1\leq i\leq m-1} \frac{1}{\omega^i-1} b_{i1}u^i + \sum_{\substack{0\leq 
i\leq m-2\\1\leq j\leq m-1}} \frac{a_{(i+1)j}}{1- \omega^j}  u^iv^j + \sum_{1\leq j\leq m-1} 
\frac{a_{0j}}{(1- \omega^j)\alpha} x^{m-1}v^j \,.$$
We see that the pure part of $\vartheta$ is completely determined by $d(u)$ and $d(v)$.
\end{proof}

\begin{definition}
The derivation $d_s \in Der ((\alpha, \beta)_{k, \omega}/ (k, \delta))$ determined by $d_s(u) =\frac{\delta(\alpha)}{m \alpha} u$ and $d_s(v) = \frac{\delta(\beta)}{m\beta} v$ is called the \textbf{standard derivation} corresponding to $\alpha, \beta$. 
\end{definition}

Note that, the standard derivation is the unique derivation  in $Der ((\alpha, \beta)_{k, \omega}/ (k, \delta))$  which preserves $k(u)$ and $k(v)$ because $k(u)$ and $k(v)$ are separable algebraic extensions of $k$, and $\delta$ extends uniquely to $k(u)$ and $k(v)$.

\begin{proposition}\label{tracediffhom} 
Let $(k,\delta)$ be a differential field. Let $A = (\alpha, \beta)_{k,\omega}$ be a symbol algebra 
and $d\in Der(A/(k,\delta))$. Then the set $A^0$ is stable under $d$.
\end{proposition}
\begin{proof}
 Let $a= \sum_{0\leq i,j\leq m-1} a_{ij}u^iv^j$.
 For any $\vartheta \in A$, we have that $\tr( \partial_{\vartheta} (a)) =0$, using \Cref{derchar}(\ref{derchar2}), we get that 
 $$\tr(d (a)) = \tr(d_s(a)) = m^2 \delta(a_{00}) = d( \tr(a)).$$
Thus trace is a differential homomorphism and the kernel of trace $A^0$ is $d$-stable. 
\end{proof}

\begin{remark}\label{tracediffhomalg}
Let $(A,d)$ be a $n$-dimensional algebra over $(k,\delta)$ and $\partial_{-d}$ be the inner derivation on $\End_k(A)$. Then $(\End_k(A, \partial_{-d})$ is differentially isomorphic to $(M_n(k), d_P)$ for some trace $0$ matrix $P$. For $Q\in M_n(k)$, $\tr (d_P(Q)) = \tr(\delta^c(Q)) = \delta(\tr(Q))$. Since left regular representation  $ A\rightarrow \End_k(A)$, given by $a \mapsto (l_a : x \mapsto ax )$ is a differential monomorphism, by transport of structure, we get that $\tr: A\rightarrow k$ is a differential homomorphism.      
\end{remark}

\section{Constants and differential subfields}

Let $(k,\delta)$ be a differential field and $(B,d)$ be a finite dimensional differential central simple algebra over $(k, \delta)$.
A subfield $L \subseteq B$ is called a \textbf{differential subfield} of $(B,d)$ if $d(L) \subseteq L$.
Note that, in this case,  $d$ restricts to the unique derivation $\delta_L$ on $L$. 
In \cite[Theorem]{Ami}, it was shown that given a central simple $k$-algebra $B$ and a maximal subfield $L$ of $B$, there exists $d\in Der(B/(k,\delta))$ such that $L$ is a differential subfield of $(B,d)$.
However it is not clear, if we are given an arbitrary differential central simple algebra $(B,d)$ over $(k, \delta)$, does it always contains a maximal differential subfield? 
In this section, we investigate this for symbol algebras.

We recall that  $C_{B,d} = \{ b\in B\mid d(b) =0 \}$ is the set of constants of $(B,d)$.
For an element $\vartheta \in B$, let $Z_B(\vartheta ) = \{\varrho \in B \mid \varrho \vartheta = \vartheta \varrho \}$ is the \textbf{centraliser of $\vartheta$ in $B$.}

For a differential symbol algebra $(A,d)$, we show that there always exists a maximal differential subfield of $A$ when $\delta$ is the zero derivation on $k$. 

\begin{theorem}\label{contantscentralizerequality}
Let $(k, \theta)$ be a differential field with zero derivation $\theta$.
Let $(A,d)$ be a differential symbol algebra over $(k, \theta)$.
Then
\begin{enumerate}[$(1)$]
\item for $\vartheta \in A$, $C_{A, \partial_\vartheta} =Z_A(\vartheta)$.
\item for a maximal subfield $k(\vartheta)$ of $A$, $C_{A,\partial_{\vartheta}} = k(\vartheta)$.
\item there exists a maximal differential subfield $L$ of $A$ such that $d(L) =0$. 
\end{enumerate}
\end{theorem}

\begin{proof}
$(1)$ Note that, for $\varrho \in A$ we have $0=\partial_{\vartheta}(\varrho) =\varrho \vartheta - \vartheta \varrho$ if and only if $\varrho  \in Z_A(\vartheta)$.

$(2)$ For a maximal subfield $k(\vartheta)$ of $A$, we have $Z_A(\vartheta) = k(\vartheta)$, hence the statement follows from $(1)$.

$(3)$ By \Cref{derchar} (\ref{derchar2}) we have that $d =d_s+\partial_\vartheta$ for some $\vartheta \in A^0$. 
Since $\theta$ is the zero derivation, $d_s$ is also the zero derivation on $A$ and hence $d = \partial_\vartheta$.
Thus, by $(1)$, the set of constants of $(A,d)$ is $Z_A(\vartheta)$.
Since $[k(\vartheta): k] \leq m$, by the Double Centraliser Theorem we have that $dim_k(Z_A(\vartheta)) \geq m$, and hence $Z_A(\vartheta)$ contains a maximal subfield $L$ of $A$. 
Since $L\subseteq Z_A(\vartheta)$ and $d(Z_A(\vartheta)) =0$, we get that $d(L) =0$, whereby $d(L) \subseteq L$.
\end{proof}

\begin{example}
Let $k$ be a number field containing a primitive $m$th root of unity.
Then the zero derivation $\theta$ is the unique derivation on $k$.
Our Theorem above guarantees that for any differential symbol algebra $(A,d)$ over $(k,\theta)$, there exists a maximal differential subfield.   
\end{example}

In view of \Cref{contantscentralizerequality}, it is interesting to construct examples of differential symbol algebras over a differential field (with a non-zero derivation) which do not have any maximal differential subfield.
Let $C$ be a field of characteristic $0$ and $C(t)$ be the rational function field in one variable $t$. 
Then the $C$-linear map $\frac{d}{dt} : C(t) \rightarrow C(t)$ determined by $\frac{d}{dt} (t^n) = nt^{n-1}$ is a derivation on $C(t)$.
Then $\frac{d}{dt}$ is a non-zero derivation on $C(t)$ and its field of constants is $C$. 
In the following proposition, assuming that $C$ contains  a primitive $m$th root of unity, we give a class of matrices $P\in M_m(C(t))$ such that the differential algebra $(M_m(C(t)), d_P)$ has no cyclic differential subfield.  
This generalises Proposition 4.7 of \cite{AKVRS}.

\begin{proposition}\label{nocyclic_differentialfields}
Let $(C, \theta)$ be a differential field with zero derivation $\theta$.
Assume that  $\car(C) =0$ and $C$ contains a primitive $m$th root of unity, say $\omega$.
Let $k= C(t)$ and $\delta =\frac{d}{dt}$ on $k$ and $A= M_m(k)$.
Let $f\in  k$ and $\lambda_1, \ldots, \lambda_m \in \mg C$ such that $\lambda_i \neq \lambda_j$ if $i\neq j$
and consider the matrix
{\small{$$ P= \left(\begin{matrix} 
                 \lambda_1       & \ldots & f\\
                \vdots  & \ddots  & \vdots \\ 
                0     & \ldots & \lambda_m  \end{matrix} 
\right)\in A.
$$}}
We have the following 
\begin{enumerate}
\item $d_P(L) \not \subset L$ for any cyclic subfield $L$ of $A$ such that $[L:k]>1$.
\item $ X \in C_{A,d_P}$ if and only if 
$X = \left(\begin{matrix} 
                 \mu_1       & \ldots & x\\
                \vdots  & \ddots  & \vdots \\ 
                0     & \ldots & \mu_m  \end{matrix}
              \right)$
where  $\mu_1,  \ldots,  \mu_m \in C$, $x\in C(t)$ such that $\delta(x) + (\lambda_m - \lambda_1) x + f(\mu_{1} -\mu_{m})=0$.
\item \label{nocyclic_differentialfields3}  If $f = \frac{f_1}{f_2}$ such that $f_1, f_2 \in C[t]$ are coprime polynomials and $f_2$ has an irreducible factor with multiplicity $1$,
then $C_{A,d_P} = \{ \diag(\mu_1,  \ldots, \mu_{m-1}, \mu_1) \mid  \mu_1, \ldots, \mu_{m-1} \in C\}$,
\end{enumerate}
\end{proposition}
\begin{proof}
$(1)$ Let $L$ be a cyclic subfield of the algebra $M_m(k)$.
Set $m_1 = [L:k]$.
Then, by Kummer Theory \cite[Corollary 4.3.9]{GilSza}, there exists $X\in M_m(k)$ such that $L=k(X)$ and $X^{m_1} = \varphi$ for some $\varphi \in \mg k \setminus {k^{\times}}^{m_1}$.
Then $X^m = \vartheta$ where $\vartheta = \varphi^{\frac{m}{m_1}} \in \mg k$. 
Since $\varphi \notin {k^{\times}}^{m_1}$, we have $\vartheta \notin {k^{\times}}^{m}$.

Suppose that $d_P(L) \subseteq L$.
Then $d_P$ must coincide with the unique extension of $\delta$ on $L$ and hence $d_P(X) = \frac{\delta(\vartheta)}{m\vartheta} X$.
On the other hand, $d_P(X) = \delta^c(X) + XP-PX $.
Thus
\begin{equation}\label{}
\delta^c(X) + XP-PX =\frac{\delta(\vartheta )}{m\vartheta } X.
\end{equation}

Writing  $X = (x_{ij}) \in M_m(k)$, and comparing the coordinates of the matrices in the above equation, we obtain
\begin{align}
\delta(x_{ij}) + (\lambda_j -\lambda_{i}) x_{ij} =\frac{\delta(\vartheta )}{m\vartheta } x_{ij}\,\, &\mbox{ for } 1 < i\leq m, 1\leq j<m, \label{ij}\\
\delta(x_{im}) + (\lambda_m - \lambda_{i}) x_{im}+ fx_{i1} =\frac{\delta(\vartheta )}{m\vartheta } x_{im} &\mbox{ for } 1 < i\leq m, \label{im}\\
\delta(x_{1j}) + (\lambda_j - \lambda_1) x_{1j} -fx_{mj} =\frac{\delta(\vartheta )}{m\vartheta } x_{1j}\,\, &\mbox{ for }  1\leq j<m,\label{1j}\\
\delta(x_{1m}) + (\lambda_m - \lambda_1) x_{1m} + f(x_{11} -x_{mm})=\frac{\delta(\vartheta )}{m\vartheta } x_{1m} &\label{1m}.
\end{align} 

For $g,h \in C[t]\setminus \{0\}$, define $\deg(\frac{h}{g}) = \deg(h)-\deg(g)$ and observe that, for any $x\in \mg {k}$ and $\mu \in \mg C$, $\deg (\delta(x))  \neq  \deg ((\frac{\delta(\vartheta)}{m\vartheta} + \mu ) x)$. 
Thus \Cref{ij} implies that $x_{ij} =0$ for $1 < i\leq m, 1\leq j<m$ and $i\neq j$. 
For $2 \leq  i,j\leq {m-1}$, \Cref{im} and \Cref{1j} become $\delta(x_{im}) +(\lambda_m - \lambda_{i}) x_{im} =\frac{\delta(\vartheta)}{m\vartheta} x_{im}$ and   $\delta(x_{1j}) + (\lambda_j - \lambda_1) x_{1j} =\frac{\delta(\vartheta)}{m\vartheta} x_{1j}$  respectively and hence again by the same argument we obtain $x_{im } =0 =x_{ij}$ for $2 \leq i,j\leq {m-1}$.
Thus, for $1\leq i,j\leq m$, if  $x_{ij} \neq x_{1m}$ and $i\neq j$, then $x_{ij} = 0$.

Now comparing the diagonal entries in the matrix equation $X^m = \vartheta 
 \in M_m(k)$ gives that $x_{ii}^m = \vartheta$ for $1\leq i\leq m$. 
In particular, $\vartheta \in {k^{\times}}^{ m}$, which is a contradiction as observed in the first paragraph. 
Thus we conclude that $d_P(L) \not\subset L$ for any cyclic subfield of $M_m(k)$.

$(2)$ Let $X = (x_{ij}) \in C_{A, d_P}$. By comparing the coordinates in the matrix equation $d_P(X) =0 $ we get the following equations:
\begin{align}
\delta(x_{ij}) + (\lambda_j -\lambda_{i}) x_{ij} &=0\,\, \mbox{ for } 1 < i\leq m, 1\leq j<m,\nonumber\\
\delta(x_{im}) + (\lambda_m - \lambda_{i}) x_{im}+ fx_{i1} &=0\,\,  \mbox{ for } 1 < i\leq m,\nonumber \\
\delta(x_{1j}) + (\lambda_j - \lambda_1) x_{1j} -fx_{mj} &=0\,\, \mbox{ for }  1\leq j<m,\nonumber\\
\label{1m0} \delta(x_{1m}) + (\lambda_m - \lambda_1) x_{1m} + f(x_{11} -x_{mm})&=0.  
\end{align}
Arguing with the degree function, the first three equations imply  that for $1\leq i,j\leq m$, if  $x_{ij} \neq x_{1m}$ and $i\neq j$, then $x_{ij} = 0$ and $\delta (x_{ii}) =0$. 
Thus for $1\leq i\leq m$, $x_{ii} \in C$ and $x_{11}, x_{mm}, x_{1m}$ must satisfy \Cref{1m0}.

$(3)$  Let $x_{1m} = \frac{g}{h}$ with $g,h \in C[t]$ such that $(g,h) =1$.
Since $f_2 \notin C$, \Cref{1m0} implies that  $h\notin C$.
Let $h = \prod p_i^{n_i}$ be a factorization of $h$ into distinct irreducible factors.
Then $H :=  \prod p_i^{n_i-1} =\gcd(h, \delta(h))$ and  let $h_1, h_2 \in C[t]$ be such that $h = Hh_1$ and $\delta(h) = Hh_2$. 
Note that $(h_1,h_2) = 1$.
Then from \Cref{1m0}, we get 
$$f_2(\delta(g)h_1  + g( (\lambda_m - \lambda_1) h_1-h_2 ))+ f_1 Hh_1^2(x_{11} -x_{mm}) =0.$$
Since $h$ is coprime to the polynomial $(\delta(g)h_1  + g( (\lambda_m - \lambda_1) h_1-h_2 )) $, we get $Hh_1^2$ divides $f_2$ and since  $f_1$ is coprime to $f_2$, we get $f_2$ divides $Hh_1^2$, whereby $f_2 = Hh_1^2$.
Since every irreducible factor of $Hh_1^2 =f_2$ occurs to at least the second power, we get a contradiction.
Hence $x_{1m} =0$ and \Cref{1m0} holds if $x_{11} =x_{mm}$. 
\end{proof}

\begin{remark}
Note that, for $m=2$ and $f = \frac{f_1}{f_2}$ such that $f_1$ and  $f_2$ are coprime and $f_2$ has an irreducible factor with multiplicity $1$ (e.g. $f =\frac{1}{t}$), using \Cref{nocyclic_differentialfields} (\ref{nocyclic_differentialfields3}) we get that $C_{A,d_P} =C$ (see \cite[Proposition 4.7]{AKVRS}).
\end{remark}

In the following theorem, we study differential symbol algebras containing a certain type of new constants.

\begin{theorem}\label{representation_constants}
Let $(k,\delta)$ be a differential field and 
$A=(\alpha,\beta)_{k,\omega}$ be a symbol algebra of degree $m$ with generators $u,v$. 
Let
$d \in Der (A/(k, \delta))$ be such that $d(k(u))\subseteq k(u)$.
Let $\vartheta = \vartheta_0+ \vartheta_1v + \dots +\vartheta_{m-1}v^{m-1}$ with $\vartheta_0, \vartheta_1,\dots, \vartheta_{m-1} \in k(u)$ be such that $d(\vartheta) =0$.
Then for  $1\leq p\leq m-1$ such that $\vartheta_p \in \mg {k(u)}$, there exists a non-zero constant $c_p\in k$ such that  the 
algebra $(\alpha,c_p\beta^p)$ splits over $k$.

In particular, if $A$ is a division algebra and  there is an element $\vartheta\in A\setminus k(u)$ with $d(\vartheta)=0$, 
then there exists a non-zero constant $c$ in $k$ and an integer $1\leq p\leq m-1$ such that the 
algebra $(\alpha,c\beta^p)$ splits over $k$.
\end{theorem}

\begin{proof}
Let $\vartheta =\sum_{0 \leq j\leq m-1}\vartheta_jv^j$, where $\vartheta_j\in k(u)$. 
Since $d(u)\in k(u)$, by Proposition \ref{derchar} $(1)$, we have 
$d(u)=\frac{\delta(\alpha)}{m\alpha}u$ and 
$d(v)=\left(\frac{\delta(\beta)}{m\beta}  +\sum_{1 \leq i\leq m-1}b_{i1}u^i\right)v$ for some $b_{i1}\in k$. 
Then
\begin{align*}
d(\vartheta)  &=  d(\vartheta_0) +\hspace{-2mm} \sum_{1 \leq j\leq m-1}\hspace{-1mm}\left( d(\vartheta_j) + \vartheta_j \Bigg( j\frac{\delta(\beta)}{m\beta} 
            + \sum_{1\leq  i \leq m-1}\hspace{-1mm}(1+\omega^i+\dots+\omega^{(j-1)i}) b_{i1}u^i\Bigg)\hspace{-1mm}\right)v^j.
\end{align*}
Since $\vartheta$ is a constant, $d(\vartheta_0)=0$ and for $1\leq j \leq m-1$ 
$$d(\vartheta_j) + \vartheta_j \Bigg( j\frac{\delta(\beta)}{m\beta} 
            + \sum_{1\leq i\leq m-1}(1+\omega^i+\dots+\omega^{(j-1)i}) b_{i1}u^i\Bigg) =0\,.$$ 
Let $p \in \{1,\ldots,m-1\}$ be such that $\vartheta_p \in \mg{ k(u)}$. 
Then
$$d(\vartheta_p)=-\Bigg( p\frac{\delta(\beta)}{m\beta} 
            + \sum_{1\leq i\leq m-1}(1+\omega^i+\dots+\omega^{(p-1)i}) b_{i1}u^i\Bigg)\vartheta_p\,.$$ 	
Let $\xi \in \ovl{k}$ be such that $\xi^m = \alpha$.
Let $\psi: k(u) \rightarrow k(\xi)$ be the $k$-algebra homomorphism determined by $u \mapsto \xi$.
Note that $ \psi \circ d = \delta_{k(\xi)}\circ \psi$. 
Let $\gamma= \psi (\vartheta_p^{-1})\in k(\xi)$. 
Then 
$$\delta_{k(\xi)}(\gamma)=\Bigg( p\frac{\delta(\beta)}{m\beta} 
+ \sum_{1\leq i\leq m-1}(1+\omega^i+\dots+\omega^{(p-1)i}) b_{i1}\xi^i\Bigg) \gamma.$$
Assume $G$ be the differential Galois group of $k(\xi)/k$ and let $s = \frac{m}{ [k(\xi): k]}$. Then 
\begin{align*}
			\delta(\text{Nr}_{k(\xi)/k}(\gamma)) &= \sum_{\tau\in G} \tau (\delta_{k(\xi)}(\gamma))\Bigg( \prod_{\sigma\in G,~  
			\sigma\neq\tau}\sigma(\gamma)\Bigg)\notag\\
			&=\sum_{\tau\in G} \tau\Bigg( p\frac{\delta(\beta)}{m\beta} 
+ \sum_{1\leq i\leq m-1}(1+\omega^i+\dots+\omega^{(p-1)i}) b_{i1}\xi^i\Bigg) \Bigg(\prod_{\sigma\in 
				G}\sigma(\gamma)\Bigg)\notag\\
			&=\text{Nr}_{k(\xi)/k}(\gamma)\text{Tr}_{k(\xi)/k}\Bigg( p\frac{\delta(\beta)}{m\beta} 
+ \sum_{1\leq i\leq m-1}(1+\omega^i+\dots+\omega^{(p-1)i}) b_{i1}\xi^i\Bigg)\notag\\
			&=\text{Nr}_{k(\xi)/k}(\gamma) \Bigg( p\frac{\delta(\beta)}{s\beta} 
+ \sum_{1\leq i\leq m-1}(1+\omega^i+\dots+\omega^{(p-1)i}) b_{i1}\text{Tr}_{k(\xi)/k}(\xi^i)\Bigg)\notag\\
			&=\text{Nr}_{k(\xi)/k}(\gamma)p\frac{\delta(\beta)}{s\beta}\,.
\end{align*}
Thus
\begin{equation*}
\frac{\delta(\text{Nr}_{k(\xi)/k}(\gamma^s))}{\text{Nr}_{k(\xi)/k}(\gamma^s)}=\frac{\delta(\beta^p)}{\beta^p}	
\qquad\text{and}\qquad \delta\bigg(\frac{\text{Nr}_{k(\xi)/k}(\gamma^s)}{\beta^p}\bigg)=0.
\end{equation*}
Therefore, there exists a nonzero constant $c$ in $k$ such that 
$\text{Nr}_{k(\xi)/k}(\gamma^s)=c\beta^p$ 
and  the algebra $(\alpha,c\beta^p)$ splits over $k$ (see \cite{GilSza}, Corollary 4.7.7).
\end{proof}

Now we consider the differential symbol algebra  $((\alpha,\beta)_{k,\omega},d_s)$ over $(k,\delta)$ and study the new constants. 
As an application of \Cref{representation_constants}, we further obtain sufficient conditions on $\alpha,\beta $ so that  $(\alpha,c\beta)_{k,\omega}$ splits over $k$ for some constant $c \in k$.

\begin{corollary}\label{constants_standard_der}
Let $(k,\delta)$ be a differential field and 
consider the differential symbol algebra $(A=(\alpha,\beta)_{k,\omega},d_s)$  of degree $m$. 
Then we have the following
\begin{enumerate}
\item $C_{k,\delta}\subsetneq C_{A,d_s}$  if and only if  $\alpha^{-i}\beta^{-j} \in C_{k,\delta}{k^{\times}}^m$ for some $0\leq i,j\leq m-1, i+j\neq0$. 
\item For $1\leq j\leq m-1$, if $\alpha^{-i}\beta^{-j} \in C_{k,\delta}{k^{\times}}^m$ for some $0\leq i\leq m-1$ then there exists $c_j \in C_{k,\delta}$ such that $(\alpha, c_j\beta^j )$ splits over $k$.   
\end{enumerate}
\end{corollary}

\begin{proof}
$(1)$ We have $d_s(u) = \frac{\delta(\alpha)}{m\alpha}u$ and $d_s(v) = \frac{\delta(\alpha)}{m\alpha}v$.
Let $\vartheta=\sum_{0 \leq i,j \leq m-1}\vartheta_{ij}u^iv^j$, where $\vartheta_{ij}\in k$ such that $d(\vartheta) =0$. 
We have
$$d(\vartheta) = \sum_{0 \leq i,j \leq m-1}\left(\delta(\vartheta_{ij})+ i\frac{\delta(\alpha)}{m\alpha}\vartheta_{ij} + j\frac{\delta(\beta)}{m\beta}\vartheta_{ij}\right)u^iv^j. $$
Thus $d(\vartheta) =0$ if and only if, for $0 \leq i,j \leq m-1$, either $\vartheta_{ij} =0$ or
$$\frac{\delta(\vartheta_{ij}^m)}{\vartheta_{ij}^m} = - \left(i\frac{\delta(\alpha)}{\alpha} + j\frac{\delta(\beta)}{\beta}\right),$$
that is 
$$ c_{ij} \vartheta_{ij}^m = \alpha^{-i}\beta^{-j} \mbox{ for some } c_{ij} \in C_{A,d_s},$$
and  observe that $c_{ij} \in k$ and hence $c_{ij} \in C_{k,\delta}$. 

We obtain that $\vartheta \in C_{A,d_s} \setminus C_{k,\delta}$ if and only if $\alpha^{-i}\beta^{-j} \in C_{k, \delta} {k^{\times}}^{ m}$ for some $0 \leq i,j \leq m-1, i+j\neq 0$.
This shows $(1)$.

$(2)$ Suppose that, for $1\leq j\leq m-1$ and $0\leq i\leq m-1$, we have  $\alpha^{-i}\beta^{-j} =c\lambda^{m} $ with $c\in C_{k,\delta}$ and $\lambda\in k$.
Then $d_s(\lambda u^iv^j) = 0$ implies that $\lambda u^iv^j \in C_{A, d_s}$. As $j\neq 0$, we get  $\lambda u^iv^j\notin k(u)$.
Since $\lambda u^i \in \mg{k(u)}$, by \Cref{representation_constants}, there exists $c_j \in C_{k,\delta}$ such that $(\alpha, c_j\beta^j )$ splits over~$k$.  
\end{proof}

\begin{example}
Let $(C, \theta)$ be a differential field with zero derivation $\theta$ (e.g.~ $C=\qq(\omega)$).
Assume that  $\car(C) =0$ and $C$ contains a primitive $m$th root of unity, say $\omega$.
Let $k= C(t)$ and $\delta =\frac{d}{dt}$ on $k$.

$(i)$ Let $f,g \in C[t]$ be  coprime and square-free.
Observe that, if  $f^{-i}g^{-j} =ch^m $ with $c\in C$ and $h\in k$ for some $0\leq i,j\leq m-1$ then $i=j=0$.
By \Cref{constants_standard_der} $(1)$, $C_{(f,g)_{k,\omega},d_s} =C_{k,\delta}$. 

$(ii)$ Let $f\in C[t]$ be square-free and $c\in \mg {C}$.
For $1\leq  r\leq m-1$, consider the algebra $A = (cf, f^r)_{k, \omega}$.
Note that, for $i =r, j=m-1$ we have that $(cf)^{-r}( f^r)^{-(m-1)} =c^{-r}f^{-mr} \in C{k^{\times}}^{\times m}$.
Note that $$d_s(x^{-r}y) =0,$$ hence $Cx^{-r}y \subseteq C_{A,d_s}\setminus k(x)$.
Then the algebra $(cf, (-c)^rf^r)_{k, \omega}$ is split.
\end{example}

\section{Differential splitting fields}
Our goal in this section is to show that an arbitrary differential symbol algebra $(A,d)$ over a differential field $(k,\delta)$ is split 
by a finitely generated field extension of $k$ and the differential symbol algebra $(A,d_s)$ is 
split by a finite field extension of $k$.

Let $A = (\alpha, \beta)_{k,\omega}$ be a symbol algebra.
We know that for  $\xi \in \ovl{k}$ such that $\xi^m = \alpha$, the field extension $k(\xi)$ 
splits the algebra $A$ in the usual sense, as $\Phi_{A, \xi} : A \otimes_k k(\xi) \rightarrow M_m(k(\xi))$ is a $k(\xi)$-algebra isomorphism (see \Cref{splittingex}).
Furthermore, there exists a matrix $P\in M_m(k(\xi))$ such that the differential 
algebras $(A \otimes_k k(\xi), d^{\ast} = d\otimes \delta_{k(\xi)} )$ and $(M_m(k(\xi)), d_P) $ over  $(k(\xi), \delta_{k(\xi)})$ are differential isomorphic (see \cite[Proposition 1]{{LJARM}}). 
In the following lemma, we first determine such a matrix $P$.

\begin{lemma}\label{partialsplit}
Let $(k,\delta)$ be a differential field. 
Let $A=(\alpha,\beta)_{k,\omega}$ be a symbol algebra and  
		$d\in Der(A/(k,\delta))$.
Let $\xi \in \ovl{k}$ be such that $\xi^m = \alpha$.
Then 
$$\Phi_{A,\xi}: (A\otimes_k k(\xi), d^\ast) \rightarrow (M_m(k(\xi)), d_P) $$
is a differential isomorphism of algebras over $(k(\xi), \delta_E)$ 
for a unique $P = (p_{rs})_{0\leq r,s\leq m-1}\in M_m(k(\xi))$ defined as follows:
$$  p_{rs} =  \left\{ \begin{array}{ll}
 \vspace{2mm}  
          \displaystyle \sum_{0 \leq i\leq m-1}\hspace{-0.2cm} a_{i(r-s)}\xi^i\omega^{(m-r)i} & \mbox{~~ if $1<r\leq m-1$ and $0\leq 
s<r$}\,,\\
     \beta \displaystyle\sum_{0 \leq i\leq m-1} \hspace{-0.2cm}a_{i(m+r-s))}\xi^i\omega^{(m-r)i}& \mbox{~~ if $1<s\leq m-1$ and 
$0\leq r<s$}\,,\\
         \displaystyle \sum_{1 \leq i\leq m-1}\hspace{-0.2cm} a_{i0} \xi^i\omega^{(m-r)i} - \displaystyle \sum_{1 \leq i\leq m-1}\hspace{-0.2cm} 
\frac{\omega^{(m-r)i}\delta(\beta)}{(\omega^i -1)m\beta} &\mbox{~~ if $0 \leq r = s\leq m-1$}\,.
        \end{array}\right. $$  
\end{lemma}

\begin{proof}
Let $d_\Phi$ be the derivation on $A \otimes_k k(\xi)$ such that $\Phi_{A,\xi}$ is a differential isomorphism of the differential algebras $(A 
\otimes_k k(\xi), d_\Phi)$ and $(M_m(k(\xi)), \delta_{k(\xi)}^c)$, i.e.~$\Phi_{A,\xi} \circ d_\Phi = \delta_{k(\xi)}^c \circ \Phi_{A,\xi}$. 
Then 
$$\Phi_{A,\xi} \circ d_\Phi( u\otimes 1) = \delta_{k(\xi)}^c(\Phi_{A,\xi} (u\otimes 1)) = \delta_{k(\xi)}^c( \mathcal A ) = \delta_{k(\xi)}(\xi) \mathcal A  = 
\frac{\delta(\alpha)}{m\alpha}\mathcal A = \Phi_{A,\xi} \left (\frac{\delta(\alpha)}{m\alpha} u \otimes 1 \right )
$$
\begin{align*}
\mbox{and }\Phi_{A,\xi} \circ d_\Phi ( v\otimes 1) = \delta_{k(\xi)}^c(\Phi_{A,\xi} (v\otimes 1)) &= \delta_{k(\xi)}^c( \mathcal B ) =  \left(
\begin{matrix}
& \delta(\beta) \\
\mathbf{0}_{m-1} &\\
\end{matrix} \right) \\ 
&=  \Phi_{A,\xi} \left (\sum_{0 \leq i \leq m-1} \frac{\delta(\beta)} {m\alpha\beta} u^i v \otimes \xi^{m-i} \right).
\end{align*}

Since $\Phi_{A,\xi}$ is an isomorphism, we conclude that 
$$
d_\Phi( u\otimes 1) = \frac{\delta(\alpha)}{m\alpha} u \otimes 1
\mbox{ and }d_\Phi( v\otimes 1) =\sum_{0 \leq i \leq m-1} \frac{\delta(\beta)} {m\alpha\beta} u^i v \otimes \xi^{m-i}.
$$
Note that, for $w = \sum_{1\leq i\leq m-1} \frac{\delta(\beta)} {m\alpha\beta (\omega^i-1)} u^i \otimes 
\xi^{m-i} $ we have $d_\Phi = d_s^\ast + \partial_w$.

Let $\vartheta \in A^0$ so that $d=d_s+\partial_\vartheta$ and $d^\ast=d_s^\ast+\partial_{\vartheta \otimes1}$.
By \cite[Proposition 1]{Ami}, there exists $\varrho \in A \otimes_k k(\xi)$ such that $d^\ast = d_\Phi+ \partial_\varrho$. 
Then
$$d_\Phi+ \partial_\varrho = d^{\ast} = d_s^{\ast} + \partial_{\vartheta \otimes 1 } = d_\Phi -\partial_w + 
\partial_{\vartheta \otimes 1 } =d_\Phi + \partial_{\vartheta \otimes 1 -w}\,,    $$
whereby $\varrho = \vartheta \otimes 1 -w$.
Note that, for $P =\Phi_{A,\xi}(\varrho)$ we have that $\Phi_{A,\xi} \circ d^{\ast} = d_P\circ\Phi_{A,\xi}$. 
It is easy to check that $P$ is as defined in the statement.
\end{proof}

\begin{corollary}
Let $(k,\delta)$ be a differential field.
Let $A=(\alpha,\beta)_{k,\omega}$ be a symbol algebra and  
		$d\in Der(A/(k,\delta))$. 
There is a differential field extension $(E, \delta_E)$ of $(k, \delta)$ such that $E$ is finitely 
generated over $k$ as a field  and $(E, \delta_E)$ splits $(A,d)$.  
\end{corollary}

\begin{proof}
The proof of~\cite[Theorem 4.1, (ii)]{AKVRS} in the quaternion case goes through for the general 
$m$.
\end{proof} We now focus on the differential algebra $(A, d_s)$.
Let $P_s$ be the matrix defined in \Cref{partialsplit} such that $(A \otimes_k k(\xi), d^\ast_s) \simeq (M_m(k(\xi), 
d_{P_s}))$ as differential algebras over $(k(\xi), \delta_{k(\xi)})$.
We now calculate the matrix $P_s$ explicitly.

\begin{lemma}\label{standard-matrix-calc} 
The matrix $P_s = \frac{ \delta(\beta)}{m\beta} \diag{(t_0, t_1, \ldots, t_{m-1})} $ with $t_{r} = \frac{m-1}{2}-r$.
\end{lemma}

\begin{proof}
Since for $d_s$, we have that for all $0\leq i,j \leq m-1$ if $a_{ij}\neq a_{10}$ and $b_{ij}\neq b_{01}$  then $a_{ij} =b_{ij} =0$,  and we get by \Cref{{partialsplit}} that 
\begin{equation}
P_s = \frac{ \delta(\beta)}{m\beta} \diag{(t_0, t_1, \ldots, t_{m-1})} \mbox{ with }  t_r = \sum_{1\leq i\leq m-1}\small{\frac{1}{\omega^{ri}(1-\omega^i)}}.
\end{equation}
For $0\leq r\leq m-2$, we have 
$$t_{r+1} -t_r = \sum_{1\leq i\leq m-1}\small{\frac{1}{\omega^{(r+1)i}}} =-1.$$
For the polynomial $f(X) = X^{m-1} + \dots +1 = \prod_{i=1}^{m-1}(X -\omega^i)$, we have that 
$$t_0 =\sum_{1\leq i\leq m-1}\small{\frac{1}{(1-\omega^i)}} = \frac{1}{f(1)}\cdot \frac{df}{dX}(1) = \frac{m-1}{2}.$$
Thus $t_r = \frac{m-1}{2}-r$.
\end{proof}

\begin{theorem}\label{split-standard-der}
Let $(k,\delta)$ be a differential field and $A = (\alpha, \beta)_{k,\omega}$ be a symbol algebra. 
 The differential algebra $(A, d_s)$ is split by a finite field extension of $k$ of degree at 
most $m^2$ if $m$ is odd and of degree at most $2m^2$ if $m$ is even.
\end{theorem}

\begin{proof}
First suppose that $m$ is odd.
Let $\xi,\eta \in \ovl{k}$ be such that $\xi^m =\alpha$, $\eta^m = \beta$ and $E = k(\xi, \eta)$.
Then the unique derivation extension $\delta_E$ of $\delta$ on $E$ is determined by $\delta_E (\xi) =  \frac{\delta(\alpha)}{m\alpha}\xi$ and $\delta_E (\eta) =  \frac{\delta(\beta)}{m\beta}\eta$.
For $0\leq r\leq m-1$, set $t_r = \frac{m-1}{2}-r$ and consider the matrix
$$F =  \diag (\eta^{t_0},\eta^{t_1}, \ldots,\eta^{t_{m-1}}) \in Gl_m(E).$$ 
Note that 
\begin{align*}
\delta_E^c(F) &= \frac{\delta(\beta)}{m\beta } \diag(t_0\eta^{t_0}, t_1\eta^{t_1}, \ldots, t_{m-1}\eta^{t_{m-1}})\\
&=  \frac{\delta(\beta)}{m\beta }  \diag(t_0, t_1, \ldots, t_{m-1}) \diag(\eta^{t_0}, \eta^{t_1}, \ldots,\eta^{t_{m-1}}).
\end{align*}
Using \Cref{standard-matrix-calc}, we get $\delta_E^c(F)  = P_sF$.
Clearly $[E: k] \leq m^2$  and  $(E, \delta_E)$ splits the differential algebra $(A,d_s)$.

Now suppose that $m$ is even.
For $0\leq r \leq m-1$, set $s_r = 2t_r$.
Let $\xi, \zeta \in \ovl{k}$ be such that $\xi^m =\alpha$ and $\zeta^{2m} = \beta$.
Set  $E = k(\xi, \zeta)$.
Then $\delta_E (\zeta) =  \frac{\delta(\beta)}{2m\beta}\zeta$.
Consider the matrix
$$F = \diag (\zeta^{s_0}, \zeta^{s_1}, \ldots, \zeta^{s_{m-1}}) \in Gl_m(E).$$ 
Then
\begin{align*}
\delta_E^c(F) &= \frac{\delta(\beta)}{2m\beta } \diag (s_0\zeta^{s_0}, s_1\zeta^{s_1}, \ldots, s_{m-1}\zeta^{s_{m-1}})\\
&=\frac{\delta(\beta)}{m\beta } \diag (t_0, t_1, \ldots, t_{m-1})\diag (\zeta^{s_0}, \zeta^{s_1}, \ldots, \zeta^{s_{m-1}}).
\end{align*} 
Using \Cref{standard-matrix-calc}, we get $\delta_E^c(F)  = P_sF$. 
Clearly $[E: k] \leq 2m^2$ and $(E, \delta_E)$ splits $(A,d_s)$.
\end{proof}

\begin{theorem}\label{zero-der-ext-split}
Let $(k, \theta)$ be a differential field with zero derivation $\theta$. 
Consider the differential symbol algebra $(A,d)$ where $ A = (\alpha , \beta)_{k, \omega}$.
\begin{enumerate}
\item \label{algsplit} If $d$ is the zero derivation, then any algebraic extension $L/k$ that splits 
the algebra $ A $, when given the zero derivation, also splits the differential symbol algebra $ (A, d)$ over $(k,\theta)$.
\item If $d$ is a non-zero derivation and $A$ is a division algebra, then any differential field 
extension $(L, \delta_L)$ that splits the differential algebra $(A, d)$ over $(k,\theta)$ is 
transcendental.  
\item Let $\varrho \in A$ be such that $k(\varrho )/ k$ is a cyclic extension of degree $m$. Then there exists a 
differential field extension $(E, \delta_E)$ of $(k, \theta)$ such that $(E, \delta_E)$ splits 
$(A,\partial_\varrho )$ and the transcendence degree of $E$ over $k$ is $m$. 
Furthermore, if 
$\varrho ^m\in k$ and $m$ is even then the field $E$ can be chosen of transcendence degree $\frac{m}{2}.$ 
\end{enumerate}
\end{theorem}

\begin{proof}
$(1)$ The proof follows similarly as the proof of  \cite[Theorem 4.5, (i)]{AKVRS} for general $m$.
	
$(2)$ Assume that $d$ is a non-zero derivation.
Note that $d_s$ is the zero derivation on $A$, and by \Cref{derchar}(\ref{derchar2}),  $d= \partial_w$ 	for some $w\in A^0$. 
 Since $w \notin k$, the center of $A$, there exists some $\vartheta\in A\setminus k(w)$ such that $\vartheta w\neq w \vartheta$, and hence 
$\partial_w(\vartheta)\neq0$. 
Let $(L, \delta_L)$ be a differential field extension  of $(k, \theta)$ that splits the differential 
algebra $(A,d)$.  
Consider a differential isomorphism $\phi :  ( A \otimes_k L, \partial_w^{\ast}) \rightarrow 
(M_m(L), \delta_L^c)$. 
Note that $\partial_w^{\ast}(\vartheta\otimes1)\neq 0$ and hence $\phi(\partial_w^{\ast}(\vartheta\otimes1))\neq 0$.

Suppose that $L/k$ is algebraic.
Since $\theta$ is the zero derivation, $\delta_L$ is the zero derivation on 
$L$.
Thus the coordinate-wise derivation $\delta_L^c$ on $M_m(L)$ is also the zero derivation,  
whereby $\delta_L^c(\phi(\vartheta\otimes1))=0$, which is a contradiction.
Therefore, we conclude that $L/k$ is transcendental.
	
$(3)$ Let $\varrho  \in A$ be such that $k(\varrho )/ k$ is a cyclic extension of degree $m$.
There exists $\varrho _1\in A$ such that  $\varrho _1^m = \alpha_1 \in k$ and $k(\varrho ) = k(\varrho _1)$.

We write $\varrho =\sum_{0 \leq i \leq m-1}a_i\varrho _1^i$ where $a_0, \ldots, a_{m-1} \in k$. 
Then there exists an element $\vartheta_1\in A$ such that $\vartheta_1\varrho _1=\omega \varrho _1\vartheta_1,\ \vartheta_1^m=\beta_1\in k$ (see 
e.g.~\cite[Lemma 3.2]{BG21}).
Then $A =(\alpha_1, \beta_1)_{k, \omega}$.
Assume $\xi\in \ovl{k}$ such that $\xi^m=\alpha_1$. 
Then $k(\xi)$ splits the algebra $A$ by the isomorphism given in \Cref{splittingex}, i.e.~ 
 $\Phi_{A, \xi} :   A \otimes_k k(\xi) 
\rightarrow M_m(k(\xi))$ determined by
\begin{center}
	$\varrho_1 \otimes 1 \mapsto \mathcal{A} = \diag(\xi,\omega^{m-1}\xi, \omega^{m-2}\xi, \ldots, \omega\xi), \,\,
	\vartheta_1\otimes 1 \mapsto \mathcal{B} = \left(
	\begin{matrix}
		 & \beta_1 \\
		I_{m-1} &
	\end{matrix} \right),$
\end{center}
and $\lambda\otimes r \mapsto \lambda r I_m$ where $\lambda \in k,r\in k(\xi)$. 
Thus, for  
$P:=\Phi_{A, \xi} (\varrho\otimes1)=\sum_{0\leq i \leq m-1}a_i\mathcal{A}^i$ we have $\Phi_{A, \xi} :   (A \otimes_k k(\xi),\partial_\varrho^{\ast}) 
\rightarrow (M_m(k(\xi)),d_P)$ is a differential isomorphism.

Consider $E=k(\xi,x_0,\dots,x_{m-1})$, the rational function field in $m$ variables $x_0,\dots,x_{m-1}$ over $k(\xi)$.
Consider the derivation $\delta_E$ on $E$ given by 
$\delta_E(x_i)=\sum_{0\leq j\leq m-1}a_j(\omega^{m-i}\xi)^jx_i$ for $0\leq i\leq m-1$.
Let 
$$ F = \diag(x_0,x_1, \ldots, x_{m-1}) \in Gl_m(E),$$ 
and observe that $\delta_E^c(F)=PF.$ Therefore, by \cite[Proposition 2]{LJARM}, we obtain 
$$(A\otimes_k E,\partial_\varrho^{\ast})\simeq(M_m(E),d_P)\simeq (M_m(E),\delta_E^c).$$
Now assume that $\varrho=\varrho_1$ and $m$ be even then the matrix $P=\phi(\varrho\otimes1)$ becomes
$$P = \left(
		\begin{matrix}
			P_0&\\
			&-P_0 
		\end{matrix} \right)\,,\quad\text{where}\quad P_0 =(\xi, -\omega^{\frac{m}{2}-1}\xi , \ldots, -\omega^2\xi, -\omega\xi).$$
Consider $E'=k(\xi,x_0,\dots,x_{\frac{m}{2}-1})$ the rational functional field in $\frac{m}{2}$ variables $x_0,\dots,x_{\frac{m}{2}-1}$ over $k(\xi)$ and $\delta_{E'}(x_i)=-\omega^{\frac{m}{2}-i}\xi x_i$ for $0\leq i\leq \frac{m}{2}-1$. 
Then, for 
$$
F =\left(
\begin{matrix}
F_0&  \\
&-F_0	
\end{matrix} \right) \in Gl_m(E'), \quad \text{where}\quad F_0 = \diag(x_0, x_1, \ldots,x_{\frac{m}{2}-1} ) \in Gl_{\frac{m}{2}}(E'),
$$
we have $\delta_{E'}^c(F)=PF.$ Therefore, for the differential field $(E', \delta_{E'})$ of 
transcendence degree $\frac{m}{2}$ over $(k,\delta)$ we obtain $(A\otimes_k 
E',\partial_u^{\ast})\simeq (M_m(E'),\delta_{E'}^c)$.
\end{proof}

\section{Maximal cyclic splitting subfields}

Every central simple algebra is split by its maximal subfields.
However, we cannot say the same for differential central simple algebras (see \cite[Theorem 4.5, Proposition 4.7]{AKVRS}).
In the following proposition we observe necessary conditions on maximal cyclic subfields $L$ of a 
symbol algebra $(\alpha , \beta)_{k, \omega}$ such that the differential field $(L, \delta_L)$ splits the differential algebra
$((\alpha , \beta)_{k, \omega}, d_s)$.

\begin{proposition}\label{max-subfield-split}
Let $(k, \delta)$ be a differential field and $m$ be a prime number.
Let $d_s$ be the standard derivation on the symbol algebra $ A = (\alpha , \beta)_{k, \omega}$ with 
$\alpha, \beta \in k \setminus C_{k, \delta}{k^{\times}}^{ m}$. 
Let $L$ be a cyclic maximal subfield of $A$ of degree $m$.
If $(L, \delta_L)$ splits the differential algebra  $(A,d_s)$ then there exist elements $\lambda , \mu 
\in C_{k,\delta}$
and $\xi, \eta \in A$ such that $L = k(\xi) = k(\eta)$ with $\xi^m =\lambda \alpha $ and $\eta^m = 
\mu \beta$.
\end{proposition}

\begin{proof}
Since $k$ contains a primitive $m$th root of unity, by Kummer Theory \cite[Corollary 4.3.9]{GilSza}, we have that 
$ L = k(\gamma)$  with $\gamma^m = \nu \in \mg k$. 
Suppose that $(L, \delta_L)$ splits the differential algebra $(A, d_s)$.
Let $\psi: ( A \otimes_k L, d_s^{\ast}) \rightarrow (M_m(L), \delta_L^c)$ be a differential isomorphism.
Let $u,v\in A$ be such that $u^m =\alpha, v^m = \beta$ and $vu = \omega uv$.
We write $\psi( u\otimes 1) = (\alpha_{ij}) \in M_m(L)$ and $\psi( v\otimes 1) = 
(\beta_{ij}) \in M_m(L)$. 

Since $\psi$ is a differential isomorphism and $d_s(u) = \frac{\delta(\alpha)}{m\alpha}u$, we get that 
\begin{align*}
\delta_L^c &\circ \psi (u\otimes 1 )  = \psi \circ d_s^{\ast} (u\otimes 1) =  
\frac{\delta(\alpha)}{m\alpha}\psi (u \otimes 1).  
\end{align*}
Substituting $\psi( u\otimes 1) = (\alpha_{ij})$ and comparing the coordinates of the matrices,  we obtain 
\begin{equation}\label{dermatcor}
\delta_L(\alpha_{ij}) =  \frac{\delta(\alpha)}{m\alpha} \alpha_{ij} \mbox{ for all }  0 \leq i,j \leq m-1\,.
\end{equation}
Consider now $i,j \in  \{0, \ldots, m-1\}$.
Since $\alpha_{ij} \in L$, we write
\begin{align*}
\alpha_{ij} &= \sum_{0 \leq r \leq m-1} a_{ij}^{(r)} \gamma^r 
\end{align*}
with $ a_{ij}^{(r)}\in k$ for $0 \leq r \leq m-1$.
Then, using \Cref{dermatcor}, we obtain
$$\delta_L(\alpha_{ij}) = \sum_{0 \leq r\leq m-1} \left(\delta (a_{ij}^{(r)}) + r a_{ij}^{(r)} 
\frac{\delta_L(\gamma)}{ \gamma}\right) \gamma^r = \frac{\delta(\alpha)}{m\alpha} \sum_{0 \leq r \leq m-1} 
a_{ij}^{(r)} \gamma^r \,.$$
Comparing the coefficients of $\gamma^r$ for $0\leq r\leq m-1$, we get 
\begin{equation}\label{eqn-cooeff-compare} \delta (a_{ij}^{(r)}) + r a_{ij}^{(r)} 
\frac{\delta_L(\gamma)}{ \gamma} = \frac{\delta(\alpha)}{m\alpha}  a_{ij}^{(r)}\,.
\end{equation}
Substituting $\frac{\delta_L(\gamma)}{\gamma}= \frac{\delta(\nu)}{m\nu}$ \Cref{eqn-cooeff-compare}, and multiplying \Cref{eqn-cooeff-compare} by $\frac{m (a_{ij}^{(r)})^{m-1}\nu^r}{\alpha}$, for $r \in \{0, \ldots, m-1\}$, we have
 $$\delta \left ( \frac{(a_{ij}^{(r)})^m \nu^r}{\alpha} \right) =  0. $$
 Thus, for $r \in \{0, \ldots, m-1\}$, we conclude that either $a_{ij}^{(r)} =0$ or   
$\frac{(a_{ij}^{(r)})^m \nu^r}{\alpha}\in  \mg{C}_{k,\delta}$.  
 
Since $(\alpha_{ij}) \in M_m(L)$ is non-zero, we have $a_{ij}^{(r)} \neq 0$ 
for some $i,j,r \in\{0, \ldots, m-1\}$. Let $i,  j, r \in \{0, \ldots, m-1\} $  such that  
$\frac{(a_{ij}^{(r)})^m \nu^r}{\alpha} = \lambda \in  \mg{C}_{k,\delta}$.
Then $ \alpha  =\frac{ (a_{ij}^{(r)})^m }{\lambda }\nu^r $.
Since $\alpha \notin C_{k, \delta}{k^{\times}}^{ m}$, we get that $r\neq 0$.
Thus $r\in \{ 1, \ldots, m-1\}$, and for $\xi = a_{ij}^{(r)} \gamma^r$ we have $\xi^m = \lambda \alpha$. 
Furthermore, since $m$ is prime, $r$ is coprime to $m$ and hence $L = k(\gamma) = k(\xi)$.

Using that $\psi$ is an isomorphism and $d_s(v) = \frac{\delta(\beta)}{m\beta}v$, analogous 
calculations show that $L = k(\eta)$ for some $\eta \in A$ and $\mu \in \mg{C}_{k,\delta}$ such that $\eta^m 
=\mu\beta $.
\end{proof}

In \cite[Proposition 4.4]{AKVRS} the above theorem was proved for $(k,\delta) =(\qq(t),\frac{d}{dt}) $ and $m=2$.

\begin{corollary}\label{max_split_particular}
Let $m$ be a prime number and $\omega\in \cc$ be an $m$th root of unity. 
Consider a differential field $(k, \delta)$ where $k=\qq(\omega)(t)$ and $\delta = \frac{d}{dt}$.
Let $d_s$ be the standard derivation on the algebra $ A = (\alpha , \beta)_{k, \omega}$ with 
$\alpha, \beta \in \zz [\omega][t]\setminus \zz [\omega]{k^{\times}}^{ m} $.
If a cyclic extension $L$ of degree $m$ splits $(A,d_s)$ then there exist elements $\lambda , \mu 
\in \zz[\omega]$ and $\xi, \eta \in A$ such that $L = k(\xi) = k(\eta)$ with $\xi^m =\lambda \alpha $ and $\eta^m = 
\mu \beta$.
\end{corollary}

\begin{proof}
Since $C_{k,\delta} = \qq(\omega)$, by \Cref{max-subfield-split}, there exist $\tilde \lambda , \tilde \mu 
\in \qq(\omega)$ and $\tilde\xi, \tilde\eta \in A$ such that $L = k(\tilde \xi) = k(\tilde\eta)$ with $\tilde \xi^m =\tilde \lambda \alpha $ and $\tilde\eta^m = 
\tilde \mu \beta$.  
Since $\qq(\omega) \subseteq Z[\omega]{\qq(\omega)^\times}^{ m}$, there exists $\lambda, \mu \in \zz[\omega]$ and $\lambda_1, \mu_1\in \qq(\omega)$ such that $\lambda =\tilde\lambda \lambda_1^m$ and $\mu = \tilde \mu\mu_1^m$. Choosing $\xi = \lambda_1\tilde\xi$ and $\eta =\mu_1 \tilde \eta $, the statement follows.
\end{proof}

We now give an application of \Cref{max_split_particular} and show that there exist differential symbol algebra which are not split by any of its maximal subfields.

\begin{example}
Let $m$ be a prime number and $\omega\in \cc$ be an $m$th root of unity. 
Consider the differential field $(k, \delta)=(\qq(\omega)(t), \frac{d}{dt})$.
Let $f,g \in \qq(\omega)[t]$ and $p \in \qq(\omega)[t]$ be irreducible such that $p$ does not divide $fg$ and $g\notin \qq(\omega){k^{\times}}^m$.
Consider the algebra $(pf,g)_{k,\omega}$.
Suppose that the differential algebra  $((pf,g)_{k,\omega}, d_s)$ is split by a maximal cyclic subfield $L$ of $(pf,g)_{k,\omega}$ of degree $m$.
Then, by \Cref{max_split_particular}, there exist $\lambda , \mu 
\in \zz[\omega]$ and $\xi, \eta \in A$ such that $L = k(\xi) = k(\eta)$ with $\xi^m =\lambda pf $ and $\eta^m = 
\mu g$.
By Kummer Theory, $(\lambda pf) (\mu g)^i \in {k^\times}^m$ for some $0\leq i\leq m-1$, which is a contradiction as, for any $0\leq i\leq m-1$, the multiplicity of $p$ in $fg^i$ is $1$ and hence $(\lambda pf) (\mu g)^i \notin {k^\times}^m$. 
Thus the differential algebra $((pf,g)_{k,\omega}, d_s)$ cannot be split by any of its maximal cyclic subfield of degree $m$.  
\end{example}

\subsection*{Acknowledgement}
The authors express their gratitude to Amit Kulshrestha and Varadharaj R.~Srinivasan for introducing this interesting topic and encouragement for this work.
We are further grateful to the referee for the valuable feedback and help in improving the manuscript, specifically for providing us a conceptual proof of \Cref{tracediffhom} for a finite dimensional algebra which is included as \Cref{tracediffhomalg}.

\bibliographystyle{amsalpha}

\end{document}